\documentclass[reqno]{amsart}

\usepackage{amsmath,amsfonts,amssymb}
\usepackage{times}
\usepackage[centertags]{amsmath}
\usepackage{amsfonts}
\usepackage{amssymb}
\usepackage{amsthm}
\usepackage{newlfont}

\newtheorem{thm}{Theorem}
\newtheorem{prop}{Proposition}
\newtheorem{lem}{Lemma}
\newtheorem{remark}{Remark}

\newtheorem{cor}{Corollary}

\begin{document}

\title{Some Generalizations of Hadamard's-Type Inequalities through Differentiability for s-Convex functions and Their Applications}

\author{Muhammad Muddassar*}\label{*}
\email{malik.muddassar@gmail.com}
\author{Muhammad Iqbal Bhatti}
\email{uetzone@hotmail.com}
\address{Department of Mathematics, University of Engineering and Technology, Lahore - Pakistan}

\date{18 March, 2012.}

\begin{abstract}
In this paper, a general form of integral inequalities of Hermite - Hadamard's type through differentiability for $s$-convex function in second sense and whose all derivatives are absolutely continuous are established. The generalized integral inequalities contributes some better estimates than some already presented. The inequalities are then applied to numerical integration and some special means.
\end{abstract}

\maketitle
{\setcounter{section}{0}}
\section{Introduction}\label{sec1}
\markboth{\underline{\hspace{3.4in} Muhammad Muddassar and M. I. Bhatti}}
{\underline{\hspace{1pt}Generalization of Hadamard-Type Inequality for s-Convex functions ...\hspace{2.1in}}}\pagestyle{myheadings}

Let $f:\emptyset\neq I\subseteq\mathbb{R}\rightarrow\mathbb{R}$ be a function defined on the interval $I$ of real numbers. Then $f$ is called convex if
\begin{equation*}
    f(t\,x+(1-t)\,y)\leq\,\,t\,f(x)\,+\,\,(1-t)\,f(y)
\end{equation*}
for all $x,y\in I$ and $t\in [0,1].$ There are many results associated with convex functions in the area of inequalities, but one of those is the classical Hermite Hadamard inequality:
\begin{equation}\label{HH}
    f\left(\frac{a+b}{2}\right)\leq\frac{1}{b-a}\int_a^b f(x) dx\leq\frac{f(a)+f(b)}{2}.
\end{equation}
This inequality gives us an estimate, from below and from above, of the average value
of the convex function $f : [a, b] \rightarrow\mathbb{R}$ for $a,b\in I,$ with $a<b.$
 \\
H. Hudzik and L. Maligranda in ~\cite{r6}, define the class of functions which are $s-$convex in the second sense. This class is defined as follows:\\ A function $f:\mathbb{R}^+\rightarrow\mathbb{R}$ is said to be $s-$convex in the second sense if for any two non-negative real numbers $x, y$, for $\lambda\in[0,1]$ and adjust $s\in(0,1]$, we have the following inequality
\begin{equation}\label{e1}
    f(\lambda\,x+(1-\lambda)\,y)\leq\,\,\lambda^s\,f(x)\,+\,\,(1-\lambda)^s\,f(y)
\end{equation}
It is seen that from (\ref{e1}), for $s=1$, $s$-convex function is convex. H. Hudzik et al. in the same paper \cite{r6}  talked about some results associating with $s-$convex functions in second sense. We find some more new results about Hadamard's inequality for $s-$convex functions in ~\cite{r4,r10,r8,r9}. Although it is seen that many important inequalities connecting with 1-convex (convex) functions given in ~\cite{r2}, but one of them is $(\ref{HH}).$\\
S. S. Dragomir et al. gave a variant of Hermite-Hadamard's inequality for $s-$convex functions in second sense in ~\cite{r5}.
\begin{thm}\label{t1}
 Let a function $f:\mathbb{R}^+\rightarrow\mathbb{R}^+$ be $s-$convex in the second sense, where $s\in(0,1],$ and $a,b\in\mathbb{R}^+$, with $a<b$. If $f\in L^{1}[a,b]$, then the following inequality holds
\begin{equation}\label{e2}
    2^{s-1}f\left(\frac{a+b}{2}\right)\leq\frac{1}{b-a}\int_a^b\,f(x)\,dx\leq\frac{f(a)+f(b)}{s+1}
\end{equation}
\end{thm}
In second inequality in $(\ref{e2})$, the constant $k=\frac{1}{s+1}$ is the most suitable. In \cite{r7}, where B. Jagers gave both the right and left bound for the constant $c(s)$ in the inequality
$$c(s)\,f\left(\frac{a+b}{2}\right)\leq\,\,\frac{1}{b-a}\,\int_a^b\,f(x)\,dx$$
and improved (\ref{e2}). He proved that
$$\frac{2^{s+1}-1}{s+2}\leq\,\,c(s)\leq\,\,2^\frac{s-1}{s+1}\left(\frac{2^{s}-1}{s}\right)^\frac{s}{s+1}\leq\,\,\frac{2^{s+1}-2^{s-1}-1}{s+1}$$
In \cite{r3,r2} S. S. Dragomir et al. discussed inequalities for differentiable and twice differentiable functions connecting with the H-H Inequality on the basis of the Lemma \ref{l2}  and Lemma \ref{l1}.\\
S. S. Dragomir et al. in ~\cite[pp. 65]{r2} generalized the lemma for $n$-time differentiable mapping and he states in this way:
\begin{lem}\label{l3}
Let $f: I \subset \mathbb{R} \rightarrow \mathbb{R}$ be $n$ times differentiable function on $I^o$ with $f^{(n)} \in \mathbf{L}^1 [a, b]$, then
\begin{eqnarray}\label{le3}
    &&\!\!\!\!\!\!\!\!\!\!\!\!\nonumber(-1)^n \int_{a}^b f(x)dx = \sum_{m=1}^n (-1)^{n-m+1}\left[ \frac{(t-a)^m-(t-b)^m}{m!}\right]f^{(m-1)}(t)
   \\&&\!\!\!\!\indent\indent\indent\indent\indent\indent+\frac{1}{n!}\left[ \int_{a}^t (x-a)^n f^{(n)}(x)dx + \int_{t}^b (x-b)^n f^{(n)}(x)dx \right]
\end{eqnarray}
\end{lem}
Here we explore Lemma\ref{l1} and \ref{l2} by different approach and then make their use for investigation for some more results which generalize the results explored by S. Hussain et. al. in ~\cite{r1}.\\
Using $n=1$ and $t=\frac{a+b}{2}$ in (\ref{le3}); we get lemma \ref{l2}.
\begin{lem}\label{l2}
Let $f:I\subseteq\mathbb{R}\rightarrow\mathbb{R}$ be differentiable function on $I^\circ,$ $a,b\in I^\circ$ with $a<b$ and $f'\in L^{1}[a,b],$ then
\begin{eqnarray*}
  f\left(\frac{a+b}{2}\right)-\frac{1}{b-a}\int_a^b\,f(x)\,dx&=& \frac{(b-a)}{4}\int_0^1\,(1-t)\,\left[f'\left(ta+(1-t)\frac{a+b}{2}\right)\right.\\&& -\left.f'\left(tb+(1-t)\frac{a+b}{2}\right)\right]\,dt
\end{eqnarray*}
\end{lem}
For $n=2$ we get from identity (\ref{le3});
\begin{eqnarray}\label{le3a}
&&\!\!\!\!\!\!\!\!\!\!\!\!\!\!\!\!\!\!\!\!\!\!\!\!\!\!\!\nonumber\frac{1}{b-a}\int_a^b f(x)dx=f(t)+\frac{(b+a)-2t}{2}f'(t)\\&& \indent\indent\indent+\frac{1}{2(b-a)}\left[\int_a^t(x-a)^2f''(x)dx+\int_t^b(x-a)^2f''dx\right]
\end{eqnarray}
For $t=a$ the above equation (\ref{le3a}) becomes;
\begin{equation}\label{e3}
\frac{1}{b-a}\int_a^b f(x)dx=f(a)+\frac{(b-a)}{2}f'(a)+\frac{1}{2(b-a)} \int_a^b(x-a)^2f''dx
\end{equation}
$t=b$ the above equation (\ref{le3a}) becomes;
\begin{equation}\label{e4}
\frac{1}{b-a}\int_a^b f(x)dx=f(b)+\frac{(b-a)}{2}f'(b)+\frac{1}{2(b-a)}\int_a^b(x-a)^2f''(x)dx
\end{equation}
By adding (\ref{e3}) and (\ref{e4}), we have
\begin{eqnarray}\label{e5}
&&\!\!\!\!\!\!\!\!\!\!\!\!\!\!\nonumber\frac{f(a)+f(b)}{2}- \frac{1}{b-a}\int_a^b f(x)dx=\frac{(b-a)}{4}\left[f'(b)- f'(a)\right] \\&&\indent\indent\indent\indent\indent -\frac{1}{4(b-a)}\int_a^b\left[(x-a)^2 + (x-b)^2\right]f''(x)dx
\end{eqnarray}
Now for $x=ta + (1-t)b$ in (\ref{e5}). We get lemma \ref{l1}
\begin{lem}\label{l1}
Let $f:I\subseteq\mathbb{R}\rightarrow\mathbb{R}$ be twice differentiable function on $I^\circ$ with $f''\in L^{1}[a,b],$ then
\begin{equation*}
    \frac{f(a)+f(b)}{2}-\frac{1}{b-a}\int_a^b\,f(x)\,dx=\frac{(b-a)^2}{2}\int_0^1\,t(1-t)\,f''(ta+(1-t)b)\,\,dt
\end{equation*}
\end{lem}
In ~\cite{r1} S. Hussain et al. use the above lemmas and give some new improvements in the right classical Hermite Hadamard inequality.
We give here definition of Beta function of Euler type which will be helpful in our next discussion, which is for $x,y>0$ defined as
$$\beta(x,y)= \frac{\Gamma(x).\Gamma(y)}{\Gamma(x+y)} = \int_0^1\,t^{x-1}\,(1-t)^{y-1}\,\,dt$$
In this paper, after this Introduction, in section \ref{Sec 2} we generalize the results for some $s-$Hermite Hadamard type inequalities for $n$-differentiable functions discussed in ~\cite{r1}. In section \ref{Sec 3} we give applications of the results from section \ref{Sec 2} for quadrature rules and in section \ref{Sec 4} we will discuss application for some special means.
\section{Inequalities For $n$-differentiable functions}\label{Sec 2}
\begin{thm}\label{t2}
Let $f: I \subset [0, \infty) \rightarrow \mathbb{R}$ be $n$-times differentiable function on $I^\circ$ such that $f^{(n)} \in L^1[a,b]$, where $a, b \in I$, $a < b$. If $|f^{(n)}|$ is s-convex on $[a,b]$ for some fixed $s \in (0, 1]$, then for every $\lambda \in [a,b]$, we have
\begin{eqnarray}\label{te2}
&&\!\!\!\!\!\!\!\!\!\!\nonumber\left|(-1)^n \int_{a}^b f(x)dx + \sum_{m=1}^n (-1)^{n-m+2}\left[\frac{(\lambda-a)^m-(\lambda-b)^m}{m!}\right] f^{(m-1)}(\lambda)\right| \\&& \indent\indent
\nonumber\leq \frac{1}{n!}\left[\beta(s+1, n+1)\left((\lambda-a)^{n+1}|f^{(n)}(a)|+(b-\lambda)^{n+1}|f^{(n)}(b)|\right)\right.\\&&\indent\indent\indent\indent\indent \left.+\beta(1, n+s+1)\left((\lambda-a)^{n+1}+(b-\lambda)^{n+1}\right)|f^{(n)}(\lambda)|\right].
\end{eqnarray}
\end{thm}
\begin{proof}
From Lemma \ref{l3}, we have
\begin{eqnarray}\label{e7}
    &&\!\!\!\!\!\!\!\!\!\!\!\!\!\!\!\!\!\!\nonumber\left|(-1)^n \int_{a}^b f(x)dx + \sum_{i=1}^n (-1)^{n-m+2}\left[ \frac{(t-a)^m-(t-b)^m}{m!}\right]f^{(m-1)}(t)\right|\\&&
    \nonumber\indent\indent\leq \frac{1}{n!}\left[\left|\int_{a}^t (x-a)^n f^{(n)}(x)\right|dx + \int_{t}^b \left|(x-b)^n f^{(n)}(x)dx \right|\right]\\&&
    \nonumber \indent\indent\indent\leq \frac{1}{n!}\left[(\lambda-a)^{n+1} \int_{0}^1 (1-t)^n |f^{(n)}(ta+(1-t)\lambda)|dt \right.\\&& \indent\indent\indent\indent\left. + (b-\lambda)^{n+1}\int_{0}^1 (1-t)^n |f^{(n)}(tb+(1-t)\lambda)|dt \right]
\end{eqnarray}
Since $|f^{(n)}|$ is $s$- convex on $[a,b]$ for $t \in [0,1]$, so\\
$$|f^{(n)}(ta + (1-t)\lambda)| \leq t^s |f^{(n)}(a)|+ (1-t)^s|f^{(n)}(\lambda)|$$
Inequality (\ref{e7}) becomes
\begin{eqnarray}\label{e8}
    &&\!\!\!\!\!\!\!\!\!\!\!\!\!\!\!\!\nonumber\left|(-1)^n \int_{a}^b f(x)dx + \sum_{i=1}^n (-1)^{n-m+2}\left[ \frac{(t-a)^m-(t-b)^m}{m!}\right]f^{(m-1)}(t)\right|\\&&
    \nonumber\indent\leq \frac{1}{n!}\left[(\lambda-a)^{n+1} \int_{0}^1 (1-t)^n \left[t^s |f^{(n)}(a)|+(1-t)^s|f^{(n)}(\lambda)|\right]dt\right.\\&& \indent\indent + \left.(b-\lambda)^{n+1}\int_{0}^1 (1-t)^n \left[t^s|f^{(n)}(b)|+(1-t)^s|f^{(n)}(\lambda)|\right]dt \right]
\end{eqnarray}
Where,
\begin{equation}\label{e9}
    \int_0^1 t^s(1-t)^n dt = \beta(s+1, n+1)
\end{equation}
\begin{equation}\label{e10}
    \int_0^1 (1-t)^{n+s} dt = \beta(1, n+s+1)
\end{equation}
By combining (\ref{e8}), (\ref{e9}) and (\ref{e10}) we get (\ref{te2}).
\end{proof}
\begin{thm}\label{t3}
Let $f: I \subset [0, \infty) \rightarrow \mathbb{R}$ be $n$-times differentiable function on $I^\circ$ such that $f^{(n)} \in L^1[a,b]$, where $a, b \in I$, $a < b$. If $|f^{(n)}|^q$ is s-convex on $[a,b]$ for some fixed $s \in (0, 1]$, and $q \geq 1$. Then for every $\lambda \in [a,b]$,
\begin{eqnarray}\label{te3}
&&\!\!\!\!\!\!\!\!\!\!\!\!\nonumber\left|(-1)^n \int_{a}^b f(x)dx + \sum_{m=1}^n (-1)^{n-m+2}\left[\frac{(\lambda-a)^m-(\lambda-b)^m}{m!}\right] f^{(m-1)}(\lambda)\right| \\&&
\nonumber \leq \frac{(n+1)^{-\frac{1}{p}}}{n!}\left[(\lambda-a)^{n+1} \left\{\beta(s+1,n+1)|f^{(n)}(a)|^q + \beta(1,n+s+1)|f^{(n)}(\lambda)|^q\right\}^{\frac{1}{q}} \right.\\&& \left.
+ (b-\lambda)^{n+1}\left\{\beta(s+1, n+1)|f^{(n)}(b)|^q + \beta(1,n+s+1)|f^{(n)}(\lambda)|^q\right\}^{\frac{1}{q}}\right]
\end{eqnarray}
\end{thm}
\begin{proof}
From Lemma\ref{l3} we have
\begin{eqnarray}\label{e11}
    &&\nonumber\!\!\left|(-1)^n \int_{a}^b f(x)dx + \sum_{m=1}^n (-1)^{n-m+2}\left[ \frac{(\lambda-a)^m-(\lambda-b)^m}{m!}\right]f^{(m-1)}(\lambda)\right|\\&&
   \indent\indent\indent\indent\indent \nonumber\leq \frac{1}{n!}\left[(\lambda-a)^{n+1}\int_{0}^1 (1-t)^n |f^{(n)} (ta+(1-t)\lambda)|dt\right. \\
    &&\indent\indent\indent\indent\indent\indent\indent \left. + (b-\lambda)^{n+1}\int_{0}^1 (1-t)^n |f^{(n)} (tb+(1-t)\lambda)|dt\right]
\end{eqnarray}
From here, let's start with first integral of right side of (\ref{e11}) and using H\"{o}lder's Integral Inequality
\begin{eqnarray*}
    &&\!\!\!\! \int_{0}^1 (1-t)^n |f^{(n)} (ta+(1-t)\lambda)|dt \\&& \indent\indent\indent\indent= \int_{0}^1 (1-t)^{n(1-\frac{1}{q})}(1-t)^{n(\frac{1}{q})}|f^{(n)}(ta+(1-t)\lambda)|dt\\&& \indent\indent\indent\indent\indent\indent \leq \left(\int_{0}^1 (1-t)^n dt\right)^{\frac{1}{p}}\left(\int_{0}^1 (1-t)^n\left|f^{(n)}(ta+(1-t)\lambda)\right|^q dt\right)^{\frac{1}{q}}
\end{eqnarray*}
where $p=\frac{q}{q-1}$.\\
Now here $\int_{0}^1 (1-t)^n dt=\frac{1}{n+1}$
and above inequality becomes
\begin{eqnarray}\label{e12}
    &&\!\!\!\!\nonumber\int_{0}^1 (1-t)^n |f^{(n)} (ta+(1-t)\lambda)|dt \\&& \indent\indent\indent\indent \leq \left(\frac{1}{n+1}\right)^{\frac{1}{p}}\left(\int_{0}^1 (1-t)^n\left|f^{(n)}(ta+(1-t)\lambda)\right|^q dt\right)^{\frac{1}{q}}
\end{eqnarray}
Since $|f^{(n)}|^q$ is $s$- convex on $[a,b]$ for $t \in [0,1]$, so\\
$$|f^{(n)}(ta + (1-t)\lambda)| \leq t^s |f^{(n)}(a)|^q+ (1-t)^s|f^{(n)}(\lambda)|^q$$
now using equation (\ref{e9}) and (\ref{e10}) in (\ref{e12}), we have
\begin{eqnarray}\label{e13}
    &&\!\!\!\!\!\!\!\!\!\!\!\!\nonumber\int_{0}^1 (1-t)^n \left|f^{(n)} (ta+(1-t)\lambda)\right|dt \\&& \leq \left(\frac{1}{n+1}\right)^{\frac{1}{p}}\left[\beta(s+1, n+1)\left|f^{(n)}(a)\right|+\beta(1, n+s+1)\left|f^{(n)}(\lambda)\right|\right]
\end{eqnarray}
Similarly, we have
\begin{eqnarray}\label{e14}
    &&\!\!\!\!\!\!\!\!\nonumber\int_{0}^1 (1-t)^n \left|f^{(n)} (tb+(1-t)\lambda)\right|dt \\&& \leq \left(\frac{1}{n+1}\right)^{\frac{1}{p}}\left[\beta(s+1, n+1)\left|f^{(n)}(b)\right|+\beta(1, n+s+1)\left|f^{(n)}(\lambda)\right|\right]
\end{eqnarray}
Using (\ref{e13}) and (\ref{e14}) in (\ref{e11}), we get (\ref{te3}).
\end{proof}
\begin{remark}
For $n=1$, $\lambda=\frac{a+b}{2}$ in (\ref{te3}), we get theorem $4$ of ~\cite{r1}.
\end{remark}
\begin{thm}\label{t4}
Let $f: I \subset [0, \infty) \rightarrow \mathbb{R}$ be $n$-times differentiable function on $I^\circ$ such that $f^{(n)} \in L^1[a,b]$, where $a, b \in I$, $a < b$. If $|f^{(n)}|^q$ is concave on $[a,b]$ for conjugate numbers $p, q$ where $q \geq 1$. Then for every $\lambda \in [a,b]$,
\begin{eqnarray}\label{te4}
&&\!\!\!\!\!\!\!\!\!\!\!\!\!\!\!\!\!\!\!\!\!\!\!\!\nonumber\left|(-1)^n \int_{a}^b f(x)dx + \sum_{m=1}^n (-1)^{n-m+2}\left[\frac{(\lambda-a)^m-(\lambda-b)^m}{m!}\right] f^{(m-1)}(\lambda)\right| \\&&
\nonumber\indent\indent\indent\indent\indent\indent\indent\indent \leq \frac{(np+1)^{\frac{-1}{p}}}{n!}\left[\left(\lambda-a\right)^{n+1}\left|f^{(n)}\left(\frac{a+\lambda}{2}\right)\right| \right.\\&& \indent\indent\indent\indent\indent\indent\indent\indent\indent\indent\indent\indent \left.+\left(b-\lambda\right)^{n+1}\left|f^{(n)}\left(\frac{b+\lambda}{2}\right)\right|\right]
\end{eqnarray}
\end{thm}
\begin{proof}
From Lemma \ref{l3}, we have
\begin{eqnarray}\label{e14a}
    &&\nonumber\!\!\left|(-1)^n \int_{a}^b f(x)dx + \sum_{m=1}^n (-1)^{n-m+2}\left[ \frac{(\lambda-a)^m-(\lambda-b)^m}{m!}\right]f^{(m-1)}(\lambda)\right|\\&&
   \indent\indent\indent\indent\indent \nonumber\leq \frac{1}{n!}\left[(\lambda-a)^{n+1}\int_{0}^1 (1-t)^n |f^{(n)} (ta+(1-t)\lambda)|dt\right. \\
    &&\indent\indent\indent\indent\indent\indent\indent \left. + (b-\lambda)^{n+1}\int_{0}^1 (1-t)^n |f^{(n)} (tb+(1-t)\lambda)|dt\right]
\end{eqnarray}
From here, let's start with first integral of right side of (\ref{e14a}) and using H\"{o}lder's Integral Inequality
\begin{eqnarray}\label{e14b}
    &&\!\!\!\!\nonumber \int_{0}^1 (1-t)^n |f^{(n)} (ta+(1-t)\lambda)|dt \\&& \nonumber\indent\indent\indent\indent \leq \left(\int_{0}^1 (1-t)^{np} dt\right)^{\frac{1}{p}}\left(\int_{0}^1 \left|f^{(n)}(ta+(1-t)\lambda)\right|^q dt\right)^{\frac{1}{q}}\\&& \indent\indent\indent\indent = \left(\frac{1}{np+1}\right)^{\frac{1}{p}}\left(\int_{0}^1 \left|f^{(n)}(ta+(1-t)\lambda)\right|^q dt\right)^{\frac{1}{q}}
\end{eqnarray}
Now using the concavity of $\left|f^{(n)}\right|^q$ on $[a, b]$. And by applying the Jensen's Integral Inequality on the second integral in the left side of the inequality (\ref{e14b}), we have
\begin{eqnarray}\label{e14c}
&&\!\!\!\!\!\!\!\!\!\!\!\!\!\!\!\!\!\!\!\!\!\nonumber \int_{0}^1 \left|f^{(n)}(ta+(1-t)\lambda)\right|^q dt \leq \left(\int_{0}^1 t^0 dt \right) \left|f^{(n)} \left( \frac{\int_{0}^1 \left(ta+(1-t)\lambda\right)dt}{\int_{0}^1 t^0 dt} \right)\right|^q \\&& \indent\indent\indent\indent\indent\indent\indent\indent\indent = \left|f^{(n)} \left(\frac{a+\lambda}{2}\right)\right|^q
\end{eqnarray}
So inequality (\ref{e14b}) can be written in this way
\begin{eqnarray}\label{e14d}
     \int_{0}^1 (1-t)^n |f^{(n)} (ta+(1-t)\lambda)|dt  \leq \left(\frac{1}{np+1}\right)^{\frac{1}{p}}\left|f^{(n)} \left(\frac{a+\lambda}{2}\right)\right|
\end{eqnarray}
Similarly we have
\begin{eqnarray}\label{e14e}
     \int_{0}^1 (1-t)^n |f^{(n)} (tb+(1-t)\lambda)|dt  \leq \left(\frac{1}{np+1}\right)^{\frac{1}{p}}\left|f^{(n)} \left(\frac{b+\lambda}{2}\right)\right|
\end{eqnarray}
Using (\ref{e14d}) and (\ref{e14e}) in (\ref{e14a}), we get (\ref{te4}).
\end{proof}
\begin{remark}
For $n=1$ and $\lambda = \frac{a+b}{2}$ in (\ref{te4}) , we get theorem $5$ of ~\cite{r1}.
\end{remark}
\begin{thm}\label{t5}
Let $f: I \subset [0, \infty) \rightarrow \mathbb{R}$ be $n$-times differentiable function on $I^\circ$ such that $f^{(n)} \in L^1[a,b]$, where $a, b \in I$, $a < b$. If $|f^{(n)}|^q$ is $s$-convex on $[a,b]$ for some fixed $s \in (0, 1]$ and for conjugate numbers $p, q$ where $q \geq 1$ with $p=\frac{q}{q-1}$, then for every $\lambda \in [a,b]$,
\begin{eqnarray}\label{te5}
&&\!\!\!\!\!\!\!\!\!\!\!\!\!\!\!\nonumber\left|\sum_{m=1}^n (-1)^{n-m+2}\left[\frac{(\lambda-a)^m-(\lambda-b)^m}{m!}\right] f^{(m-1)}(\lambda) + (-1)^n \int_{a}^b f(x)dx\right| \\&&
\nonumber\indent\indent\indent \leq \frac{(np+1)^{\frac{-1}{p}}}{n!}\left(\!\frac{1}{s+1}\!\right)^{\frac{1}{q}}\!\left[\left(\lambda-a\right)^{n+1}\!\left(\left|f^{(n)}(a)\right|^q \!+\! \left|f^{(n)}(\lambda)\right|^q\right)^{\frac{1}{q}} \right.\\&& \indent\indent\indent\indent\indent\indent\indent\indent\indent \left.+\left(b-\lambda\right)^{n+1}\!\left(\left|f^{(n)}(b)\right|^q\! +\! \left|f^{(n)}(\lambda)\right|^q\right)^{\frac{1}{q}}\! \right]
\end{eqnarray}
\end{thm}
\begin{proof}
From Lemma (\ref{l3}) we have
\begin{eqnarray}\label{e15}
    &&\nonumber\!\!\left|(-1)^n \int_{a}^b f(x)dx + \sum_{m=1}^n (-1)^{n-m+2}\left[ \frac{(\lambda-a)^m-(\lambda-b)^m}{m!}\right]f^{(m-1)}(\lambda)\right|\\&&
    \nonumber\indent\indent\indent \nonumber\leq \frac{1}{n!}\left\{\int_{a}^t \left|(x-a)^n f^{(n)} (x)\right|dx + \int_{t}^b \left|(x-t)^n f^{(n)} (x)\right|dx\right\}\\&&
   \indent\indent\indent\indent\indent \nonumber = \frac{1}{n!}\left[(\lambda-a)^{n+1}\int_{0}^1 (1-t)^n |f^{(n)} (ta+(1-t)\lambda)|dt\right. \\
    &&\indent\indent\indent\indent\indent\indent\indent \left. + (b-\lambda)^{n+1}\int_{0}^1 (1-t)^n |f^{(n)} (tb+(1-t)\lambda)|dt\right]
\end{eqnarray}
From here, let's start with first integral of right side of (\ref{e15}) and using H\"{o}lder's Integral Inequality
\begin{eqnarray}\label{e16}
    &&\!\!\!\!\nonumber \int_{0}^1 (1-t)^n |f^{(n)} (ta+(1-t)\lambda)|dt \\&& \indent\indent\indent\indent \leq \left(\int_{0}^1 (1-t)^{np} dt\right)^{\frac{1}{p}}\left(\int_{0}^1 \left|f^{(n)}(ta+(1-t)\lambda)\right|^q dt\right)^{\frac{1}{q}}
\end{eqnarray}
Here $\int_{0}^1 (1-t)^{np} dt=\frac{1}{np+1}$\\
 Now using the $s$-convexity of $\left|f^{(n)}\right|^q$ on $[a, b]$ in the second integral on the left side of the inequality (\ref{e16}), we have
\begin{eqnarray}\label{e17}
&&\!\!\!\!\!\!\!\!\!\!\!\!\!\!\!\!\!\!\!\!\!\nonumber \int_{0}^1 \left|f^{(n)}(ta+(1-t)\lambda)\right|^q dt \leq \int_{0}^1 \left(t^s \left|f^{(n)}(a)\right|^q + (1-t)^s \left|f^{(n)}(\lambda)\right|^q \right)dt \\&& \indent\indent\indent\indent\indent\indent\indent\indent\indent \leq \frac{1}{s+1} \left(\left|f^{(n)}(a)\right|^q + \left|f^{(n)}(\lambda)\right|^q \right)
\end{eqnarray}
The inequality in (\ref{e16}) becomes
\begin{eqnarray}\label{e18}
     \int_{0}^1 (1-t)^n |f^{(n)} (ta+(1-t)\lambda)|dt  \leq \left(\frac{1}{np+1}\right)^{\frac{1}{p}}\left(\frac{\left|f^{(n)}(a)\right|^q + \left|f^{(n)}(\lambda)\right|^q}{s+1}\right)^{\frac{1}{q}}
\end{eqnarray}
In similar way we can prove
\begin{eqnarray}\label{e19}
     \int_{0}^1 (1-t)^n |f^{(n)} (tb+(1-t)\lambda)|dt \leq \left(\frac{1}{np+1}\right)^{\frac{1}{p}}\left(\frac{\left|f^{(n)}(b)\right|^q + \left|f^{(n)}(\lambda)\right|^q}{s+1}\right)^{\frac{1}{q}}
\end{eqnarray}
Using (\ref{e18}) and (\ref{e19}) in (\ref{e15}), we get (\ref{te5}).
\end{proof}
\begin{remark}
For $n=1$ and $\lambda = \frac{a+b}{2}$ in (\ref{te5}), we get theorem $6$ of ~\cite{r1}.
\end{remark}
\begin{cor}(Midpoint's type Inequality)\label{cr1}
In Theorem \ref{t5}, if we select $\lambda = \frac{a+b}{2}$ for $n=1$, we obtain
\begin{eqnarray}\label{cre1}
&&\!\!\!\!\!\!\!\!\!\nonumber\left|f\left(\frac{a+b}{2}\right)-\frac{1}{b-a}\int_a^b f(x)dx\right|\leq \frac{(b-a)^2}{4(p+1)^{\frac{1}{p}}}\!\left(\!\frac{1}{s+1}\!\right)^{\frac{1}{q}}\!\! \left[\left(\left|f'(a)\right|^q+\left|f'\left(\frac{a+b}{2}\right)\right|^q\right)^{\frac{1}{q}}\right.\\&&\indent\indent\indent\indent\indent\indent\indent\indent\indent\indent\indent\indent\indent \nonumber \left.+\left(\left|f'(b)\right|^q+\left|f'\left(\frac{a+b}{2}\right)\right|^q\right)^{\frac{1}{q}}\right] \\&&\indent\indent\indent\indent\indent\indent\indent\indent\indent\indent\indent\indent \leq\frac{(b-a)^2}{2(p+1)^{\frac{1}{p}}}\left(\frac{1}{s+1}\right)^{\frac{1}{q}} \left(\left|f'(a)\right|+\left|f'(b)\right|\right)
\end{eqnarray}
\end{cor}
\begin{proof}
Proof is very similar to the above theorem and at the end, the second inequality is found using the following inequality $\sum_{m=1}^n \left(\alpha_m + \beta_m\right)^r \leq \sum_{m=1}^n \left(\alpha_m\right)^r + \left(\beta_m\right)^r$ for $0\leq r<1$, where $\alpha_1, \alpha_2, \alpha_3, ..., \alpha_n , \beta_1, \beta_2, \beta_3, ..., \beta_n \geq 0$.
\end{proof}
\begin{thm}\label{t6}
Let $f: I \subset [0, \infty) \rightarrow \mathbb{R}$ be $n$-times differentiable function on $I^\circ$ such that $f^{(n)} \in L^1[a,b]$, where $a, b \in I$, $a < b$. If $|f^{(n)}|^q$ is $s$-convex on $[a,b]$ for some fixed $s \in (0, 1]$ and for conjugate numbers $p, q$ where $q \geq 1$ with $p=\frac{q}{q-1}$, then
\begin{eqnarray}\label{te6}
&&\!\!\!\!\!\!\!\!\!\!\!\!\!\!\!\nonumber\left|\sum_{m=1}^n (-1)^{n-m+2} \frac{(b-a)^m}{m!}\left[f^{(m-1)}(b)- (-1)^m f^{(m-1)}(a)\right]  + 2(-1)^n \int_{a}^b f(x)dx\right| \\&&
\!\!\!\!\!\!\!\!\!\!\!\!\nonumber \leq \frac{(b-a)^{n+1}}{n!}\left(\!\!\frac{1}{n+1}\!\!\right)^{\frac{q}{q-1}}\!\!\!\left[\left\{\beta(s+1, n+1)\left|f^{(n)}(a)\right|^q \!\!\!+ \beta(1, n+s+1)\left|f^{(n)}(b)\right|^q\right\}^{\frac{1}{q}} \right.\\&& \left.\indent\indent\indent + \left\{\beta(s+1, n+1)\left|f^{(n)}(b)\right|^q+ \beta(1, n+s+1)\left|f^{(n)}(a)\right|^q\right\}^{\frac{1}{q}}  \right]
\end{eqnarray}
\end{thm}
\begin{proof}
Let's start with the lemma \ref{l3}
\begin{eqnarray}\label{t6a}
    &&\nonumber\!\!\!\!\!\!\!\!\!\!\!\!\left|(-1)^n \int_{a}^b f(x)dx + \sum_{m=1}^n (-1)^{n-m+2}\left[ \frac{(\lambda-a)^m-(\lambda-b)^m}{m!}\right]f^{(m-1)}(\lambda)\right|\\&&
    \indent\indent \leq \frac{1}{n!}\left\{\int_{a}^\lambda \left|(x-a)^n f^{(n)} (x)\right|dx + \int_{\lambda}^b \left|(x-t)^n f^{(n)} (x)\right|dx\right\}
\end{eqnarray}
For $\lambda=a$ in (\ref{t6a}), we have
\begin{eqnarray}\label{t6b}
    &&\nonumber\!\!\!\!\!\!\!\!\!\!\!\!\!\!\!\!\!\!\!\!\!\!\left|(-1)^n \int_{a}^b f(x)dx + \sum_{m=1}^n (-1)^{n-m+2} \frac{(-1)^{m+1}(b-a)^m}{m!}f^{(m-1)}(a)\right|\\&&
    \indent\indent\indent\indent\indent\indent\indent\indent \leq \frac{1}{n!}\left\{\int_{a}^b \left|(x-b)^n f^{(n)}(x)\right|dx\right\}
\end{eqnarray}
For $\lambda=b$ in (\ref{t6a}), we have
\begin{eqnarray}\label{t6c}
    &&\nonumber\!\!\!\!\!\!\!\!\!\!\!\!\!\!\!\!\!\!\!\!\!\!\left|(-1)^n \int_{a}^b f(x)dx + \sum_{m=1}^n (-1)^{n-m+2} \frac{(b-a)^m}{m!}f^{(m-1)}(b)\right|\\&&
    \indent\indent\indent\indent\indent\indent\indent\indent \leq \frac{1}{n!}\left\{\int_{a}^b \left|(x-a)^n f^{(n)} (x)\right|dx \right\}
\end{eqnarray}
Now combining (\ref{t6b}) and (\ref{t6c}), we have
\begin{eqnarray}\label{t6d}
    &&\nonumber\!\!\!\!\!\!\!\!\!\!\!\!\!\!\!\!\!\!\!\!\!\!\left|2(-1)^n \int_{a}^b f(x)dx + \sum_{m=1}^n (-1)^{n-m+2} \frac{(b-a)^m}{m!}\left[f^{(m-1)}(b)-(-1)^mf^{(m-1)}(a)\right]\right|\\&&
    \indent\indent\indent\indent\indent\indent\indent\indent \leq \frac{1}{n!} \int_{a}^b \left\{\left|(x-a)^n\right|+\left|(x-b)^n\right|\right\} \left|f^{(n)}(x)\right|dx
\end{eqnarray}
we can write (\ref{t6d}) in the form
\begin{eqnarray}\label{t6e}
    &&\nonumber\!\!\!\!\!\!\!\!\!\!\!\!\!\!\!\!\!\!\!\!\!\!\left|2(-1)^n \int_{a}^b f(x)dx + \sum_{m=1}^n (-1)^{n-m+2} \frac{(b-a)^m}{m!}\left[f^{(m-1)}(b)-(-1)^mf^{(m-1)}(a)\right]\right|\\&&
    \nonumber\indent\indent\indent\indent\indent\indent \leq \frac{(b-a)^{n+1}}{n!}\left\{ \int_{0}^1 (1-t)^n \left|f^{(n)}(ta+(1-t)b)\right|dt \right.\\&& \left. \indent\indent\indent\indent\indent\indent\indent\indent\indent\indent\indent\indent + \int_{0}^1 t^n \left|f^{(n)}(ta+(1-t)b)\right|dt\right\}
\end{eqnarray}
Now here we can easily prove the following result
\begin{eqnarray}\label{t6f}
    &&\nonumber\!\!\!\!\!\!\!\!\!\!\!\!\!\!\!\!\!\! \int_{0}^1 (1-t)^n \left|f^{(n)}(ta+(1-t)b)\right|dt \leq \left(\frac{1}{n+1}\right)^{\frac{1}{p}}\left(\beta(s+1, n+1)\left|f^{(n)}(a)\right|^q \right. \\&& \left. \indent\indent\indent\indent\indent\indent\indent\indent\indent\indent\indent\indent +\beta(1, n+s+1)\left|f^{(n)}(b)\right|^q\right)^{\frac{1}{q}}
\end{eqnarray}
And furthermore, we have
\begin{eqnarray}\label{t6g}
    &&\nonumber\!\!\!\!\!\!\!\!\!\!\!\!\!\!\!\!\!\! \int_{0}^1 t^n \left|f^{(n)}(ta+(1-t)b)\right|dt \leq \left(\frac{1}{n+1}\right)^{\frac{1}{p}}\left(\beta(n+s+1, 1)\left|f^{(n)}(a)\right|^q \right. \\&& \left. \indent\indent\indent\indent\indent\indent\indent\indent\indent\indent\indent\indent +\beta(n+1, s+1)\left|f^{(n)}(b)\right|^q\right)^{\frac{1}{q}}
\end{eqnarray}
Using (\ref{t6f}) and (\ref{t6g}) in (\ref{t6e}), we get (\ref{te6}).
\end{proof}
\begin{remark}
For $n=2$ in (\ref{te6}), we get theorem $8$ of ~\cite{r1}.
\end{remark}
\begin{cor}(Trapezoidal's type Inequality)\label{cr2}
In Theorem \ref{t6}, if we select $n=2$ for $s=1$, we obtain
\begin{eqnarray}\label{cre2}
&&\!\!\!\!\!\!\!\!\!\nonumber\left|\frac{f(a)+f(b)}{2}-\frac{1}{b-a}\int_a^b f(x)dx\right|\leq \frac{(b-a)^2}{2(6)^{\frac{1}{p}}}\!\left(\!\frac{1}{12}\!\right)^{\frac{1}{q}}\!\! \left[\left|f''(a)\right|^q+\left|f''(a)\right|^q\right]^{\frac{1}{q}} \\&&\indent\indent\indent\indent\indent\indent\indent\indent\indent\indent \leq\frac{(b-a)^2}{2(6)^{\frac{1}{p}}}\!\left(\!\frac{1}{12}\!\right)^{\frac{1}{q}} \left[\left|f''(a)\right|+\left|f''(b)\right|\right]
\end{eqnarray}
\end{cor}
\begin{proof}
Proof is very similar as we did in corollary \ref{cr1}.
\end{proof}
\begin{thm}\label{t7}
Let $f: I \subset [0, \infty) \rightarrow \mathbb{R}$ be $n$-times differentiable function on $I^\circ$ such that $f^{(n)} \in L^1[a,b]$, where $a, b \in I$, $a < b$. If $|f^{(n)}|^q$ is $s$-concave on $[a,b]$ for some fixed $s \in (0, 1]$ and for conjugate numbers $p, q$ where $q \geq 1$ with $p=\frac{q}{q-1}$, then for every $\lambda \in [a,b]$,
\begin{eqnarray}\label{te7}
&&\!\!\!\!\!\!\!\!\!\!\!\!\!\!\!\nonumber\left|\sum_{m=1}^n (-1)^{n-m+2}\left[\frac{(\lambda-a)^m-(\lambda-b)^m}{m!}\right] f^{(m-1)}(\lambda) + (-1)^n \int_{a}^b f(x)dx\right| \\&&
\nonumber\indent\indent\indent\indent\indent \leq \frac{(np+1)^{\frac{-1}{p}}}{n!}.2^{\frac{s-1}{q}}\left[\left(\lambda-a\right)^{n+1}\left|f^{(n)}\left(\frac{a+\lambda}{2}\right)\right| \right.\\&&\left.\indent\indent\indent\indent\indent\indent\indent\indent\indent\indent\indent\indent+ \left(b-\lambda\right)^{n+1}\left|f^{(n)}\left(\frac{b+\lambda}{2}\right)\right|\right]
\end{eqnarray}
\end{thm}
\begin{proof}
To prove, we carry on in similar way as we did in theorem \ref{t5}.\\
By $s$-concavity of $\left|f^{(n)}\right|^q$ we obtain
\begin{eqnarray}\label{e7a}
&&\!\!\!\!\!\!\!\!\!\!\!\!\!\!\!\!\!\!\!\!\! \int_{0}^1 \left|f^{(n)}(ta+(1-t)\lambda)\right|^q dt \leq 2^{s-1} \left|f^{(n)}\left(\frac{a+\lambda}{2}\right)\right|^q
\end{eqnarray}
In an analogous manner
\begin{eqnarray}\label{e7b}
&&\!\!\!\!\!\!\!\!\!\!\!\!\!\!\!\!\!\!\!\!\! \int_{0}^1 \left|f^{(n)}(tb+(1-t)\lambda)\right|^q dt \leq 2^{s-1} \left|f^{(n)}\left(\frac{b+\lambda}{2}\right)\right|^q
\end{eqnarray}
From (\ref{e15}), (\ref{e7a}) and (\ref{e7b}) instantly give (\ref{te7}).
\end{proof}
\begin{remark}
For $n=1$ and $\lambda= \frac{a+b}{2}$ in (\ref{te7}), we get theorem $7$ of ~\cite{r1}.
\end{remark}
\begin{thm}\label{t8}
Let $f: I \subset [0, \infty) \rightarrow \mathbb{R}$ be $n$-times differentiable function on $I^\circ$ such that $f^{(n)} \in L^1[a,b]$, where $a, b \in I$, $a < b$. If $|f^{(n)}|^q$ is concave on $[a,b]$ for conjugate numbers $p, q$ where $q \geq 1$ with $p=\frac{q}{q-1}$, then
\begin{eqnarray}\label{te8}
&&\!\!\!\!\!\!\!\!\!\!\!\nonumber\left|\sum_{m=1}^n (-1)^{n-m+2} \frac{(b-a)^m}{2m!}\left[f^{(m-1)}(b)- (-1)^m f^{(m-1)}(a)\right]  + (-1)^n \int_{a}^b f(x)dx\right| \\&&
\indent\indent\indent\indent\indent\indent\indent\nonumber \leq \frac{(b-a)^{n+1}}{n!}\left(\frac{1}{np+1}\right)^{\frac{1}{p}}\left|f^{(n)}\left(\frac{a+b}{2}\right)\right|\\&&
\indent\indent\indent\indent\indent\indent\indent\indent= \frac{(b-a)^{n+1}}{n!}\left(\beta(np+1,1)\right)^{\frac{1}{p}}\left|f^{(n)}\left(\frac{a+b}{2}\right)\right|
\end{eqnarray}
\end{thm}
\begin{thm}\label{t9}
Let $f: I \subset [0, \infty) \rightarrow \mathbb{R}$ be $n$-times differentiable function on $I^\circ$ such that $f^{(n)} \in L^1[a,b]$, where $a, b \in I$, $a < b$. If $|f^{(n)}|^q$ is $s$-convex on $[a,b]$ for conjugate numbers $p, q$ where $q \geq 1$ with $p=\frac{q}{q-1}$, then
\begin{eqnarray}\label{te9}
&&\!\!\!\!\!\!\!\!\!\!\!\nonumber\left|\sum_{m=1}^n (-1)^{n-m+2} \frac{(b-a)^m}{2m!}\left[f^{(m-1)}(b)- (-1)^m f^{(m-1)}(a)\right]  + (-1)^n \int_{a}^b f(x)dx\right| \\&&
\nonumber \leq \frac{(b-a)^{n+1}}{n!}\left(\frac{1}{np+1}\right)^{\frac{1}{p}}\left(s+1\right)^{-\frac{1}{q}}\left(\left|f^{(n)}(a)\right|^q+\left|f^{(n)}(b)\right|^q\right)^{\frac{1}{q}}
\\&&
\indent= \frac{(b-a)^{n+1}}{n!}\left(\beta(np+1,1)\right)^{\frac{1}{p}}\left(s+1\right)^{-\frac{1}{q}}\left(\left|f^{(n)}(a)\right|^q+\left|f^{(n)}(b)\right|^q\right)^{\frac{1}{q}}
\end{eqnarray}
\end{thm}
\begin{thm}\label{t10}
Let $f: I \subset [0, \infty) \rightarrow \mathbb{R}$ be $n$-times differentiable function on $I^\circ$ such that $f^{(n)} \in L^1[a,b]$, where $a, b \in I$, $a < b$. If $|f^{(n)}|^q$ is $s$-concave  on $[a,b]$ for conjugate numbers $p, q$ where $q \geq 1$ with $p=\frac{q}{q-1}$, then
\begin{eqnarray}\label{te10}
&&\!\!\!\!\!\!\!\!\!\!\!\nonumber\left|\sum_{m=1}^n (-1)^{n-m+2} \frac{(b-a)^m}{2m!}\left[f^{(m-1)}(b)- (-1)^m f^{(m-1)}(a)\right]  + (-1)^n \int_{a}^b f(x)dx\right| \\&&
\indent\indent\indent\indent\indent\indent\indent\nonumber \leq \frac{(b-a)^{n+1}}{n!}\left(\beta(np+1,1)\right)^{\frac{1}{p}}2^{\frac{s-1}{q}}\left|f^{(n)}\left(\frac{a+b}{2}\right)\right| \\&&
\indent\indent\indent\indent\indent\indent\indent\indent= \frac{(b-a)^{n+1}}{n!}\left(\frac{1}{np+1}\right)^{\frac{1}{p}}\!\!2^{\frac{s-1}{q}}\left|f^{(n)}\left(\!\frac{a+b}{2}\!\right)\right|
\end{eqnarray}
\end{thm}
\section{Applications to Composite Quadrature Rules}\label{Sec 3}
Let $K$ be the partition $\{a= x_0 < x_1 < ... < x_{n-1} < x_n = b\}$ of the interval $[a, b]$ and consider the quadrature formula
\begin{equation}\label{m1}
\int_a^b f(x)dx = S(f, K) + R(f, K)
\end{equation}
where
\begin{equation*}
S(f, K)=\sum_{m=0}^{n-1} f\left(\frac{x_m + x_{m+1}}{2}\right)\left(x_{m+1}-x_m\right)
\end{equation*}
for the midpoint version and $R(f,K)$ denotes the related approximation error.
\begin{equation*}
S(f, K)=\sum_{m=0}^{n-1} \frac{f(x_m)+f(x_{m+1})}{2}\left(x_{m+1}-x_m\right)
\end{equation*}
for the trapezoidal version and $R(f,K)$ denotes the related approximation error.
\begin{prop}\label{pro1}
Let $f: I \subseteq \mathbb{R} \rightarrow \mathbb{R}$ be a differentiable mapping on $I^o$ such  that $f' \in L^1[a, b]$, where $a, b \in I$ with $a < b$ and $|f'|$ is $s$-convex on $[a, b]$, then
\begin{eqnarray}\label{pr1}
&&\!\!\!\!\!\!\!\!\!\!\!\!\!\!\!\!\left|R(f, K)\right|\! \leq \!\frac{1}{(p+1)^{\frac{1}{p}}}\left(\!\frac{1}{s+1}\!\right)^{\frac{1}{q}}\!\sum_{m=0}^{n-1} \frac{\left(x_{m+1}-x_m\right)^3}{2}\left[\left|f'(x_m)\right|\!+\left|f'(x_{m+1})\right|\right]
\end{eqnarray}
\end{prop}
\begin{proof}
By applying subdivisions $[x_m, x_{m+1}]$ of the division $k$ for $m=0, 1, 2, ..., n-1$ on Corollary \ref{cr1}, we have \begin{eqnarray}\label{pre1}
&&\!\!\!\!\!\!\!\!\!\!\!\!\!\!\!\!\!\!\!\!\!\!\!\!\!\!\!\!\!\nonumber\left|\frac{1}{x_{m+1} - x_m} \int_{x_m}^{x_{m+1}}f(x)dx - f\left(\frac{x_{m+1}+x_m}{2}\right)\right| \\&&  \indent\indent\indent\indent\indent\leq\frac{(x_{m+1}-x_m)^2}{2(p+1)^{\frac{1}{p}}}\left(\frac{1}{s+1}\right)^{\frac{1}{q}}\left(\left|f'(x_{m+1})\right|+\left|f'(x_m)\right|\right)
\end{eqnarray}
Taking sum over $m$ from $0$ to $n-1$ and taking into account that $|f'|^q$ is $s$-convex, we get
\begin{eqnarray}\label{pre2}
&&\!\!\!\!\!\!\!\!\!\!\!\!\!\!\!\nonumber\left|\int_a^b \!\!f(x)dx - \!S(f, K)\right|  =\left|\!\sum_{m=0}^{n-1}\left\{\!\!\int_{x_m}^{x_{m+1}}f(x)dx\! - \! f\left(\!\frac{x_{m+1}+x_m}{2}\right)\!\left(x_{m+1}-x_m\right)\right\}\right|\\&&\indent\indent\indent\indent\indent\indent\nonumber \leq \!\sum_{m=0}^{n-1}\left|\left\{\!\!\int_{x_m}^{x_{m+1}}f(x)dx\! - \!\!\left(x_{m+1}-x_m\right) f\left(\!\frac{x_{m+1}+x_m}{2}\right)\right\}\right|
\\&& \indent\indent\indent\indent\indent\indent\nonumber  =  \!\sum_{m=0}^{n-1}\left(x_{m+1}-x_m\right)\left|\left\{\!\frac{1}{\left(x_{m+1}-x_m\right)}\!\int_{x_m}^{x_{m+1}}f(x)dx \right.\right.\\&& \indent\indent\indent\indent\indent\indent\indent\indent\indent\indent\indent\indent\indent\indent\indent \left.\left. - f\left(\!\frac{x_{m+1}+x_m}{2}\right)\right\}\right|
\end{eqnarray}
By combining (\ref{pre1}) and (\ref{pre2}), we get (\ref{pr1}). Which completes the proof.
\end{proof}
\begin{prop}\label{pro2}
Let $f: I \subseteq \mathbb{R} \rightarrow \mathbb{R}$ be a twice differentiable mapping on $I^o$ such  that $f'' \in L^1[a, b]$, where $a, b \in I$ with $a < b$ and $|f''|$ is $s$-convex on $[a, b]$, then
\begin{eqnarray}
&&\!\!\!\!\!\!\!\!\!\!\!\!\!\!\nonumber\left|R(f, K)\right| \leq \frac{1}{(6)^{\frac{1}{p}}}\left(\!\frac{1}{(s+2)(s+3)}\!\right)^{\frac{1}{q}}\!\sum_{m=0}^{n-1} \frac{\left(x_{m+1}-x_m\right)^3}{2}\left[\left|f''(x_m)\right|^q+\left|f''(x_{m+1})\right|^q\right]^{\frac{1}{q}} \\&&
\indent \leq \!\frac{1}{(6)^{\frac{1}{p}}}\left(\!\frac{1}{(s+2)(s+3)}\!\right)^{\frac{1}{q}}\!\sum_{m=0}^{n-1} \frac{\left(x_{m+1}-x_m\right)^3}{2}\left[\left|f''(x_m)\right| +\left|f''(x_{m+1})\right|\right]
\end{eqnarray}
\end{prop}
\begin{proof}
Proof is very similar as that of Proposition \ref{pro1}  by using corollary \ref{cr2}.
\end{proof}
\section{Applications to Some Special Means}\label{Sec 4}
Let us recall the following means for any two positive numbers $a$ and $b$.
\begin{enumerate}
  \item  \textit{The Arithmetic mean}
  $$A\equiv A(a,b)=\frac{a+b}{2}$$
  \item \textit{The Harmonic mean}
  $$H\equiv H(a,b)=\frac{2ab}{a+b}  $$
  \item \textit{The $p-$Logarithmic mean}\\
  $L_p\equiv L_{p}(a,b)=\left\{
                           \begin{array}{ll}
                             a, & \hbox{if $a=b$;}   \\
                             \left[\frac{b^{p+1}-a^{p+1}}{(p+1)(b-a)}\right]^\frac{1}{p}, & \hbox{if $a\neq b$.}
                           \end{array}
                         \right.$
  \item The $Identric\ \ mean$\\
  $I\equiv I(a,b)=\left\{
                           \begin{array}{ll}
                             a, & \hbox{if $a=b$;}  \\
                             \frac{1}{e}\left(\frac{b^b}{a^a}\right)^\frac{1}{b-a}, & \hbox{if $a\neq b$.}
                           \end{array}
                         \right.$
 \item \textit{The Logarithmic mean}\\
  $L\equiv L(a,b)=\left\{
                           \begin{array}{ll}
                             a, & \hbox{if $a=b$;}   \\
                             \frac{b-a}{\ln b\ -\ \ln a}, & \hbox{if $a\neq b$.}
                           \end{array}
                         \right.$
\end{enumerate}
The following inequality is well known in the literature in ~\cite{r9}:
$$H\leq G \leq L\leq I\leq A$$
It is also known that $L_p$ is  monotonically increasing over $p\in\mathbb{R},$ denoting $L_0=I$ and $L_{-1}=L$.\\
Now Here we find some new applications for special means of real numbers by using the results of Section 2.
\begin{prop}\label{P1}
Let $p>1,$ $0<a<b$ and $q=\frac{p}{p-1}$. Then one has the inequality.
\begin{equation}\label{S1}
\left| A(a,b)\ - \ L(a,b)\right| \leq \frac{(b-a)^2}{3} A \left(|a|, |b|\right)
\end{equation}
\end{prop}
\begin{proof} By Theorem \ref{t6} applied for the mapping $f(x)=e^x$ for $s=1$ we have the above inequality (\ref{S1}).
\end{proof}\\
A result which is connected with Geometric, Identric and Harmonic mean is the following one:
\begin{prop}\label{P2}
Let $p>1,$ $0<a<b$ and $q=\frac{p}{p-1},$ then
\begin{equation*}
\left|\frac{\mathrm{G}(a, b)}{\mathrm{I}(a, b)}\right| \leq \exp \left[-\frac{(b-a)^2}{2}\left(\frac{2}{p+1}\right)^2 \,\,\,\mathrm{H}^{-1}\left(a,b\right)\right]
\end{equation*}
\end{prop}
\begin{proof} Follows by Theorem \ref{t5}, setting $f(x)=- \ln(x)$ for $n=1$ and $s=1.$
\end{proof}\\
More results which are connected with $p-$Logarithmic mean $L_{p}(a,b)$ is the following one:
\begin{prop}\label{P3}
Let $p>1,$ $0<a<b$ and $q=\frac{p}{p-1},$ then
\begin{equation*}
\left|\mathrm{A}^{\frac{1}{2}}(a, b)-\mathrm{L}_p^2(a, b)\right| \leq  \frac{(b-a)^2}{2(p+1)^{\frac{1}{p}}}\left(\frac{1}{2}\right)^{\frac{1}{q}} \,\,\,\mathrm{H}^{-1}\left(a^{\frac{1}{2}},b^{\frac{1}{2}}\right)
\end{equation*}
\end{prop}
\begin{proof} Follows by of Corollary \ref{cr1} of Theorem \ref{t5}, setting $f(x)=\sqrt{x}, x\geq 0$ for $n=1$ and $s=1.$
\end{proof}
\begin{prop}\label{P4}
Let $p>1,$ $0<a<b$ and $q=\frac{p}{p-1}$, then
\begin{eqnarray*}
&&\!\!\!\!\!\!\!\! \left|A\left[(1-a)^n, (1-b)^n\right] - L_n^n\left[(1-a), (1-b)\right]\right| \\&& \leq \frac{(b-a)^2}{12^{\frac{q-1}{q}}}\left(\frac{n(n-1)}{(s+2)(s+3)}\right)^{\frac{1}{q}} \left(A^{1/q}\left[|1-a|^{q(n-1)},|1-b|^{q(n-1)}\right]\right)
\\&& \leq \frac{(b-a)^2}{12^{\frac{q-1}{q}}}\left(\frac{n(n-1)}{(s+2)(s+3)}\right)^{\frac{1}{q}} \left(A\left[|1-a|^{(n-1)},|1-b|^{(n-1)}\right]\right)
\end{eqnarray*}
\end{prop}
\begin{proof} Follows by Theorem \ref{t6}, setting $f(x)=(1-x)^n$, $|n| \geq 2$ and $n \in \mathbb{Z}$.
\end{proof}
\section{Conclusion}\label{sec4}
Here we can further find some new relations in the same way as above associating with some special means by taking some other convex functions. For example choosing different convex functions like $f(x)= -\ln x$, $f(x)=\frac{1}{x}$ and $f(x)=-\ln (1-x)$ for different values of $s$ in $s$-convexity(concavity), we get new relations relating to to some special means.


\begin{thebibliography}{99}
\bibitem{r4} M. Alomari and M. Darus, Hadamard-type inequalities for $s-$convex functions, \emph{Inter. Math. Forum}, 3(40) (2008) 1965-1970.
\bibitem{r10} M. Alomari and M. Darus, On Co-ordinated $s-$convex functions, \emph{Inter. Math. Forum}, 3(40) (2008) 1977-1989.
\bibitem{r3} S. S. Dragomir, On some inequalities for differentiable convex functions and applications, (submitted)
\bibitem{r5} S. S. Dragomir, S. Fitzpatrick, The Hadamard's inequality for $s-$convex functions in the second sense, \emph{Demonstratio Math}., 32 (4) (1999) 687-696.
\bibitem{r2} S. S. Dragomir, C. E. M. Pierce, \emph{Selected Topics on Hermite-Hadamard Inequalities and Applications}. RGMIA, Monographs, Victoria University 2000. (online: http://ajmaa.org/RGMIA/monographs.php/).
\bibitem{r6} H. Hudzik, L. Maligranda, Some remarks on $s-$convex functions, \emph{Aequationes math.}, 48 (1994) 100-111.
\bibitem {r1} S. Hussain, M. I. Bhatti and M. Iqbal, Hadamard-Type inequalities for s-Convex Functions I, \emph{Punjab University Jouranl of Mathematics} (ISSN 1016-2526) Vol. 41 (2009) pp51-60.
\bibitem{r7} B. Jagers, On a Hadamard-type inequality for $s-$convex functions. http://wwwhome.cs.utwente.nl/~jagersaa/alphaframes/Alpha.pdf.
\bibitem{r8} U. S. Kirmaci, M. Klari\v ci\'c Bakula, M. E. \"{O}zdemir, J. Pe\v cari\' c, Hadamard-type inequalities for $s$-convex functions, \emph{Appl. Math.Comput.}, 193(1), (2007) 26-35. doi:10.1016/j.amc.2007.03.030.
\bibitem{r9} J. Pe\v cari\' c, F. Proschan, Y. L. Tong, {\it Convex functions,partial orderings and statistical applications}, Academic Press, Inc., New York, 1992, p. 137.
\end{thebibliography}
\end{document}